\documentclass[12pt, reqno]{amsart}
\usepackage{amsmath, amsthm, amscd, amsfonts, amssymb, graphicx, color, hyperref}
\usepackage{blkarray}
\newtheorem{theorem}{Theorem}[section]

\newtheorem{proposition}[theorem]{Proposition}
\newtheorem{corollary}[theorem]{Corollary}
\theoremstyle{definition}

\theoremstyle{remark}
\newtheorem{remark}[theorem]{Remark}
\numberwithin{equation}{section}

\begin{document}
\setcounter{page}{1}

\title[Harnack parts]{Harnack parts for some truncated shifts}
\author[G. Cassier, M. Benharrat]{Gilles Cassier$^{1*}$, Mohammed Benharrat$^2$}
\address{$^{1}$ Universit\'e de Lyon 1; Institut Camille Jordan CNRS UMR 5208; 43, boulevard du 11 Novembre 1918, F-69622 Villeurbanne.
}

\email{\textcolor[rgb]{0.00,0.00,0.84}{cassier@math.univ-lyon1.fr}}
\address{$^{2}$ D\'epartement de Math\'ematiques et informatique, Ecole Nationale Polytechnique d'Oran-Maurice Audin;  B.P. 1523  Oran-El M'Naouar, Oran, Alg\'erie.
}
\email{\textcolor[rgb]{0.00,0.00,0.84}{mohammed.benharrat@gmail.com, benharrat@math.univ-lyon1.fr}}




\date{ 02/12/2019.
\newline \indent $^{*}$ Corresponding author}
\begin{abstract}
The purpose of this paper is to analysis the Harnack part of some truncated shifts whose $\rho$-numerical radius equal one in the finite dimensional case. As pointed out in Theorem 1.17
\cite{CaBeBel2018}, a key point is to describe the null spaces of the $\rho$-operatorial kernel of these truncated shifts. We establish two fundamental results in this direction and some applications are also given.
\end{abstract}
 \maketitle
\section{Introduction and preliminaries}

Shift operators are one of the most important models in operator theory. On the one hand it is an interesting class of operators in providing examples and  counter-examples  to  illustrate  many  properties in operator theory. On the other hand, these operators  becomes   a fundamental building block in the structure theory of Hilbert space operators.
For instance, they play a fundamental role in dilation theory (see \cite{SzN}) and also in the Jordan canonical form of matrix in the finite dimensional case. They also appear in the constants 
involved in constrained von Neumann inequalities (see \cite{BaCa}).  Also it is used as a mathematical tool in several areas (quantum mechanics, control theory,\dots ). For  a good survey of properties of the shift, we can see   \cite{Nikolski86, Shields74} and the references therein. 

The purpose of this paper is  to establish a new role of shift operators in  the class of the $\rho$-contractions. More precisely, we study the Harnack parts of the following so called $w_{\rho}$-normalized truncated shift $S$ of size $n+1$ defined in the canonical basis by
\begin{equation}\label{truncated shift}
S= S_{n+1}(a)
\end{equation} 
where
$$  S_{n+1}(b)  =\begin{pmatrix}
0&b& & &0\\
&0&\ddots& &\\
& &\ddots& &b \\
0& & & & 0\\
\end{pmatrix}.
$$
and $a=w_\rho(S_{n+1})^{-1}=w_\rho(S_{n+1}(1))^{-1}$ ($w_\rho(.)$ is the $\rho$-numerical radius defined bellow). For $\rho=2$ this sequence is completely described by U. Haagerup and P. De la Harpe in \cite{HaaDe92} and given by
$$((\cos\dfrac{\pi}{n+1})^{-1})_{n\geq 2}.$$
Also the same authors gave more attention of the importance of this remarkable sequence by saying that: " Its recent popularity is due the work of V. Jones on index of subfactors and all that \cite{Jones83}, but it appears in many domains, such as elementary geometry of regular polygons \cite[Formula 2.84]{Coxeter61}, graph theory or Fuchsian groups \cite{GoDeJon89}, to mention but a few", see \cite[p. 379]{HaaDe92}. Unfortunately, we don't have  an explicit formula for $a_k(\rho)$ when $\rho\neq 2$, but this ubiquitous sequence is fundamental in order to analysis the class of $\rho$-contractions, in particular this work shows the crucial role of this sequence in the description of the Harnack part of $w_{\rho}$-normalized truncated shift $S$.

Before to give our mains results, we recall some facts about the class of $\rho$-contractions as well as the Harnack relation  in the general  Hilbert spaces setting.
Let $H$ be a complex Hilbert space and  $B(H)$ be the Banach algebra of all
bounded linear operators on  $H$. For $\rho>0$,  an operator $T\in B(H)$ is called  a \emph{ $\rho$-contraction} if $T$ admits a {\it{unitary $\rho$ dilation}} (see for instance
 \cite{SzNF1} and \cite{SzNFBK}). It means that there is a Hilbert space $\mathcal{H}$ containing $H$ as a closed subspace and a unitary operator $U\in B(\mathcal{H})$ such that 
\begin{equation}\label{eq:rhodil}
T^n=\rho P_H U^n|H,\quad
n\in\mathbb{N}^{\ast},
\end{equation}
where $P_H$ is the orthogonal projection onto the subspace $H$ in $\mathcal{H}$. Denote 
by $C_\rho (H)$ the set of all $\rho$-contractions. In connection with this, J.~A.~R.~Holbrook \cite{H1} and J.~P.~Williams \cite{W},
independently, introduced  the 
$\rho$-\emph{numerical radius}
(or the \emph{operator radii} ) of an operator $T\in B(H)$ by setting
\begin{equation}\label{eq:1-rad}
w_\rho(T):=\inf\{ \gamma>0:\frac{1}{\gamma}T\in C_\rho (H)\}.
\end{equation}
 Also, $T\in C_\rho (H)$ if and only if $w_\rho(T)\leq 1$. In particular, when $\rho=1$  and $\rho = 2$, this definition reduces to contractions and operators whose numerical ranges are contained in the closed unit disc (see \cite{Be}), respectively.

Let $T_0$, $T_1$ in $C_\rho (H)$, recall that $T_1$ is Harnack dominated by $T_0$ (see
\cite{CaSu}) if there exists a constant $c \geq 1$ such that
$$
\Re p(T_1)\leq c^{2} \Re p(T_0)+ (c^{2}-1)(\rho -1)\Re p(0_{H})
$$
for any polynomial $p$ with $\Re p \geq 0$ on $\overline{\mathbb{D}}$, where $\Re z $ is the real part of a complex number $z$.

The following operatorial $\rho$-kernel (see \cite{Cassier_1, CaF_1, CaF_2})
\begin{equation}
K_{z}^{\rho} (T) =(I-\overline{z}T)^{-1}+(I-zT^{*})^{-1}+ (2-\rho)I, \quad (z\in \mathbb{D}),
\end{equation}
associated with any bounded operator $T$ having its spectrum in the closed unit disc $\overline{\mathbb{D}}$, plays a central role in Harnack analysis of operators (see for instance \cite{Cassier_2, CaSu, CaBeBel2018}). It is essentially due
to the fact that it allows us to use directly harmonic analysis methods in this setting. 

The $\rho$-kernels are related to $\rho$-contraction by the next result. An operator $T$ is in the class $C_\rho (H)$ if and only if, $\sigma(T)\subseteq \overline{\mathbb{D}}$ and
$K_{z}^{\rho} (T)\geq 0$ for any $z\in \mathbb{D}$  (see \cite{CaF_2}). For $ T_{1}, T_{0}\in C_\rho (H)$; it is proved in \cite{CaSu} that $T_{1}$ is Harnack dominated by $T_{0}$, denoted by $T_{1}\stackrel{H}{\prec} T_{0}$, if and only if $T_{1}$ and $T_{0}$ satisfy
\begin{equation}\label{harnack}
K_{z}^{\rho} (T_{1}) \leq c^{2} K_{z}^{\rho} (T_{0}) \quad \text{ for all } z\in \mathbb{D}
\end{equation}
 for some constant $c \geq 1$. A detailed description
of these  Harnack domination and other  equivalent definitions are given in \cite[ Theorem 3.1]{CaSu}.

 The relation $\stackrel{H}{\prec}$ is a preorder relation (reflexive and transitive) in $C_\rho (H)$ and  induces an  equivalent relation,  called Harnack equivalence. The associated equivalence classes are called the Harnack parts of $C_\rho (H)$. So, we say that $T_{1}$ and $T_{0}$ are Harnack equivalent and  we write $T_{1}\stackrel{H}{\sim} T_{0}$, if they belong to the same Harnack parts.  Classifying the equivalence classes induced by this preorder relation  is  a complicated question and is an important topic discussed by many authors. In \cite[Theorem 4.4]{CaSu}; it is shown  that the Harnack parts of $C_\rho (H)$ containing the null operators $O_H$  is exactly the class  of all strict $\rho-$contractions; that is  a  $\rho-$contraction such that $w_{\rho} (T)<1$. This result extends the earlier work   of  Foia\c{s}   ,\cite{Foias}, from contraction class to  $\rho$-contractions. An interesting question is now to describe the Harnack parts of $\rho$-contractions $T$ with $\rho$-numerical radius one. A  few answers in the literature of the previous question are given, essentially  in the class of contractions  with norm one.  In \cite{AnSuTi}, the authors have proved that if $T$ is either isometry or coisometry contraction then the Harnack part of $T$ is trivial (i.e. equal to $\{T\}$), and if $T$ is compact or $r(T)<1$, or normal and nonunitary, then its Harnack  part is not trivial in general. 
  It was proved in \cite{KSS} that the Harnack part of a contraction $T$ is trivial  if and only if $T$ is an isometry or a coisometry (the adjoint of an isometry), this  a response of the question posed by  Ando et al. in the class  of contractions. The authors  of \cite{BaTiSu} proved that maximal elements for the  Harnack domination in $C_1(H)$ are precisely the singular unitary operators and the minimal elements are isometries and 
 coisometries. 
  
  Recently in \cite{CaBeBel2018}, it is proved that if $T_{0}$ is a compact operator
  (i.e. $T_{0} \in \mathcal{K}(H)$) with $ w_\rho(T_{0})=1$ and with no spectral values in the torus $\mathbb{T}$, then a $\rho$-contraction $T_{1} \in \mathcal{K}(H)$ with $\sigma(T_1)\cap \mathbb{T}=\emptyset$, is Harnack equivalent to $T_{0}$ if and only if $K_{z}^{\rho} (T_{0})$ and $K_{z}^{\rho} (T_{1})$ have the same null space for all $z\in\mathbb{T}$ ({\it{null spaces condition}}) and satisfy a {\it{conorms condition}} (see  \cite[Theorem 1.17]{CaBeBel2018}). Moreover, when the dimension of the null space  of $K_{z}^{\rho} (T_{0})$ is constant over $\mathbb{T}$, the {\it{null spaces condition}} is necessary and sufficient. As a corollary,  if $T_{0}$ is a compact contraction with $ \Vert T_{0}\Vert=1$  and $r(T_{0})<1$, then a contraction
  $T_{1} \in \mathcal{K}(H)$ is Harnack equivalent to $T_{0}$ if and only if $I-T_{0}^*T_{0}$ and
  $I-T_{1}^*T_{1}$ have the same kernel and $T_{0}$ and $T_{1}$ restricted to the kernel of $I-T_{0}^*T_{0}$ coincide. A nice application is the  description  of the Harnack part of the (nilpotent) Jordan block of size $n+1$ (which is the $w_{1}$-normalized truncated shift of size $n$) as a contraction. More generally, it seems that $w_{\rho}$-normalized truncated shift of size $m$ play a crucial role in Harnack analysis of class $C_{\rho} (H)$.

  Further, the Harnack part of the $w_{2}$-normalized truncated shift is also given in two  dimensional and  three dimensional case.  Surprising, in the first case the Harnack part is trivial, while in the  second case the Harnack part is an orbit associated with the action of a group of unitary  diagonal matrices. We notice that the structure of the Harnack part of the normalized truncated shift are of different nature when $\rho=1$ and $\rho=2$ and, in the later case,  depend on the parity of dimension. According to this fact, it is  naturally  to ask about the structure of the  Harnack part of the $w_{2}$-normalized truncated shift in the case of the dimension is more than three, or more generally of the $w_{\rho}$-normalized truncated shift.

Our purpose in this article is to investigate  the Harnack part of  the $w_{\rho}$-normalized truncated shift $S$ (section 3). This will be  done by describing the null space of the $\rho$-operatorial kernel of any elements in the Harnack part of $S$ (section 2).  We conclude the paper with some open problems related to our results.

\section{Null space of operatorial kernels of
truncated shifts}\label{sec:2}
Let $S$ the $w_{\rho}$-normalized truncated shift in $C_{\rho}(\mathbb{C}^{n+1})$ defined  by \eqref{truncated shift}. Clearly, $w_\rho(S)=1$ (which justifies the term "$w_{\rho}$-normalized").  Let $T \in C_{\rho}(\mathbb{C}^{n+1})$  in the  Harnack part of  $S$,  as said before  $K_{z}^{\rho} (S)$ and $K_{z}^{\rho} (T)$ have the same null space for all $z\in\mathbb{T}$,   see \cite[Theorem 1.17]{CaBeBel2018}. The next result describes more precisely the null space of the operatorial kernels of all elements of the Harnach part of $S$. 
\begin{theorem}\label{shift}Let $\rho>1$ and  $T \in C_{\rho}(\mathbb{C}^{n+1})$ is Harnack equivalent to $S$. Then 
$$\mathcal{N} (K_{z}^{\rho} (T))=\mathbb{C}(v_0,zv_1, \dots,z^n v_n),\quad  \text{ for all }z\in \mathbb{T},$$
with $v_0\neq 0$, and $v_k= - v_{n-k}$. 
\end{theorem}
To prove this theorem the following result is needed,
\begin{theorem}\cite[Theorem 1.2]{CaBeBel2018}\label{spectral+torus}
	Let $T_0 , T_1 \in C_{\rho}(H)$, ($\rho \geq 1$), if $T_1 {\stackrel{H}{\prec}} T_0$ then $\sigma(T_1)\cap \mathbb{T} \subseteq \sigma(T_0)\cap \mathbb{T}$.
\end{theorem}
\begin{proof}Since  $T \stackrel{H}{\sim} S$, then by  \eqref{harnack}, there exist $c\geq 1$ such that 
\begin{equation}\label{Inequ:Harnack}\frac{1}{c^2}K_{z}^{\rho} (S) \leq  K_{z}^{\rho} (T)\leq c^2 K_{z}^{\rho} (S), \quad \text{ for all } z\in \mathbb{D}.
\end{equation}
Since $\sigma(S)\cap \mathbb{T}$ is empty,  by Theorem \ref{spectral+torus} the operator $S$ does not admit eigenvalues in $\mathbb{T}$. Hence, $K_{z}^{\rho} (T)$ and  $ K_{z}^{\rho}(S)$ are uniformly bounded in $\mathbb{D}$ and may be extended to a positive operators on  $\overline{\mathbb{D}}$ and by \eqref{Inequ:Harnack}\
$$\mathcal{N} (K_{z}^{\rho} (T))=\mathcal{N} (K_{z}^{\rho} (S)) \text{ for all } z\in \mathbb{T}.$$
Since
$$ K_{z}^{\rho} (S) =(I-zS^{*})^{-1}[ \rho I+2(1-\rho) Re (\overline{z}S)+(\rho-2) \left|z\right|^{2} S^{*}S](I-\overline{z}S)^{-1},$$
we have $\dim(\mathcal{N} (K_{z}^{\rho} (T)))=\dim(\mathcal{N} (L_{z}^{\rho} (S))$ for all  $z\in \mathbb{T}$ where
$ L_{z}^{\rho} (S)=\rho I+2(1-\rho) Re (\overline{z}S)+(\rho-2) \left|z\right|^{2} S^{*}S$. Thus,
$$  L_{z}^{\rho} (S)  =
\left(
\begin{array}{cccccc}
\rho  & (1-\rho)\overline{z}a     &0 &\cdots &\cdots           &0 \\
(1-\rho)az             &\rho+(\rho-2)a^2& (1-\rho)\overline{z}a       &\ddots &               &\vdots \\
0           &      (1-\rho)az   & \ddots       &\ddots &\ddots           &\vdots\\
\vdots           &  \vdots               & \ddots       &\ddots &(1-\rho)\overline{z}a          &0\\
\vdots           &   \vdots              &              &\ddots &\rho+(\rho-2)a^2&(1-\rho)a\overline{z}\\
0                &      0    &  \dots       &0   &        (1-\rho)az    &\rho+(\rho-2)a^2
\end{array}
\right)
$$
 If $(x_0, \dots,x_n)\in \mathcal{N} (L_{z}^{\rho} (S)$, then 
 $$\mathcal{S}_1\begin{cases} \rho x_0 + (1-\rho)z a x_1= 0\cr
 (1-\rho)z a x_0 + (\rho+(\rho-2)a^2)x_1 +(1-\rho)\overline{z}a x_2 =0 \cr
\vdots \cr
 (1-\rho)z a x_{n-2} + (\rho+(\rho-2)a^2)x_{n-1} +(1-\rho)\overline{z}a x_n =0\cr
(1-\rho)z a x_{n-1} + (\rho+(\rho-2)a^2)x_{n} =0.
\end{cases} $$ 
 Since $\rho >1$, we can see that a solution of $(\mathcal{S}_1)$  is completely determined by knowledge of  $x_0$, so  the dimension of  $\mathcal{N}( L_{z}^{\rho} (S))$ is at most one for all $z\in \mathbb{T}$ and $\rho > 1$.
 
Let $V(z)=v_0 (z)e_0 +\dots +v_n (z)e_n \neq 0$ in $\mathcal{N}( K_{z}^{\rho} (S))$  for all $z\in \mathbb{T}$. Then
$$\mathcal{S}(z)\begin{cases} \rho v_0 (z) +  a\overline{z} v_1(z)+ \dots +a^{n-1}\overline{z}^{n-1} v_{n-1}(z) +a^n\overline{z}^n v_n(z)= 0\cr
 a z v_0 (z) +  \rho  v_1(z)+ \dots +a^{n-2}\overline{z}^{n-2} v_{n-1}(z)  +a^{n-1}\overline{z}^{n-1} v_n(z)= 0\cr
\vdots \cr
  a^{n} z^{n} v_0 (z) +  \rho  v_1(z)+ \dots+ a z v_{n-1}(z) + \rho v_n(z)=0
\end{cases} \quad (z \in \mathbb{T}).$$

Multiplying the $(k+1)$-th equation by $\overline{z}^{k}$  and putting $w_k(z)=\overline{z}^{k} v_{k}(z)$ for $k=0\dots n$, we get
$$  a^{k}w_0 (z) +  a^{k-1} w_1(z)+ \dots + a w_{k-1}(z) + \rho w_{k}(z) +a w_{k+1}(z)+ \dots + a^{n-k} w_n(z)= 0, $$
for $k=0\dots n$ and for every $z\in \mathbb{T}$. Thus $W(z)=(w_0(z),w_1(z), \dots, w_n (z))$ is a solution of $\mathcal{S}(1)$. On the other hand, since $w_\rho(S)=1$, there exists a $z_0 \in \mathbb{T}$ such that $\mathcal{N} (K_{z}^{\rho} (S))\neq \{0\}$. Hence the dimension of  $\mathcal{N}( K_{z}^{\rho} (S))$ at least one. We derive that $\dim (\mathcal{N} (K_{z}^{\rho} (S)))=1$. Furthermore, there exists $\alpha (z)$ such that $W(z)=\alpha (z)W(1)$, with $\alpha (z)\neq 0$ for all $z\in \mathbb{T}$. We put $v_k =v_k(1)$ and $V=V(1)=W(1)$, then we have $w_k =\alpha (z) v_k$ and $v_k(z)=\alpha (z) z^kv_k$. This implies that $\mathcal{N} \left( K_{z}^{\rho} (T)\right) =\mathcal{N} \left( K_{z}^{\rho} (S)\right)   =\mathbb{C}(v_0,zv_1, \dots,z^n v_n)$ with $v_0 \neq 0$. If we assume that $v_0 =0$, by using the fact that $V$ is a solution of $\mathcal{S}(1)$, step by step we  get $v_1= \dots =v_n =0$, this yields to a contradiction. Now, let $\mathcal{W}e_k=e_{n-k}$, we have $$\mathcal{W}(v_0, v_1, \dots, v_n)=(v_n, v_{n-1}, \dots, v_0)=\alpha (v_0, v_1, \dots, v_n), $$
because $\mathcal{W}(V)$ remains a solution of $\mathcal{S}(1)$.  Since $\mathcal{W}^2 =I$, we deduce  that $V$ is an eigenvector of $\mathcal{W}$ and $\alpha =\pm 1$. We conclude that $v_k=\varepsilon v_{n-k}$, with $\varepsilon =\pm 1$.

Now, to complete the proof of theorem   we prove that $\epsilon=-1$. According to the  parity of the dimension, we distinguish two cases. 

If  $n=2p-1$, assume that $v_{n-k}=v_k$ for $k=0, \dots, p-1$. Then
$$\mathcal{S}(1)\begin{cases} \rho v_0  +  a v_1+ \dots +a^{p-1}v_{p-1} +a^p v_{p-1} \dots + a^{2p-1}v_{2p-1}= 0\cr
 a v_0  + \rho  v_1+ \dots +a^{p-2}v_{p-1} +a^{p-1} v_{p-1} \dots + a^{2p-2}v_{2p-1}= 0\cr
\vdots \cr
 a^{2p-1} v_0  + a^{2p-2}  v_1+ \dots +a^{p-1} v_{p-1} +a^{p} v_{p-1} \dots + \rho v_{2p-1}= 0\cr
\end{cases},$$
or equivalently $M\tilde{V}=0$, with
$$ 
M  =\begin{pmatrix}
\rho+ a^{2p-1} & a+a^{2p-2}& \dots & a^{p-1}+a^{p}\\
a + a^{2p-2} &  \rho+a^{2p-3}& \dots & a^{p-2}+a^{p-1}\\
\vdots &\vdots & \dots &\vdots \\
a^{p-1}+a^{p}  &  a^{p-2}+a^{p-1} & \dots & \rho + a
\end{pmatrix} \quad  \text{ and } \quad
 \tilde{V}  =\begin{pmatrix}
v_{0}\\
v_{1}\\
\vdots \\
v_{p-1}
\end{pmatrix}
 $$
 on $\mathbb{C}^{p}$.
But $M=K_{1}^{\rho} (\tilde{S})+N$, with 
$$ \tilde{S} = \begin{pmatrix}
            0 & a & 0 & \cdots  & 0 \\
            0 & 0 & a & \cdots  & 0 \\
            \vdots & \vdots& \ddots & \ddots & \vdots \\
            0 & 0 & 0 & \ddots  & a \\
            0 & 0 & 0 & \cdots  & 0
        \end{pmatrix}\quad \text{ and } \quad  N  =\begin{pmatrix}
 a^{2p-1} & a^{2p-2}& \dots & a^{p}\\
 a^{2p-2} & a^{2p-3}& \dots & a^{p-1}\\
\vdots &\vdots & \dots &\vdots \\
a^{p}  &  a^{p-1} & \dots &  a
\end{pmatrix}.
 $$
 The matrix $N$ is a hermitian with rank one and $\sigma_p(N)=\{0, Tr(N)\}$ ($Tr(N)>0$), so $N$ is positive.
 Further, $w_\rho(\tilde{S})=w_\rho(a S_p)=\frac{w_\rho( S_p)}{w_\rho( S_{2p})}<1,$ then $K_{1}^{\rho} (\tilde{S})$ is positive and invertible. This implies that $M$ is also invertible. So, $v_0 =v_1= \dots =v_{p-1} =0$ and hence $v_{2p-1} =v_{2p-2}= \dots =v_{p} =0$. We conclude that $V=0$, a contradiction. We deduce that if $n+1$ is an even number then we have necessarily $\varepsilon=-1$. 
 
 Now, assume that $n=2p$, we have $v_{2p-k}=\varepsilon v_k$ for $k=0, \dots, p$. In this case, the system $(\mathcal{S}(1))$ is equivalent to
$$0=\sum_{k=0}^{p}v_kK_{1}^{\rho} (\tilde{S})e_k+\varepsilon  (U \otimes \tilde{U}) \tilde{V}, $$
with 
$$ 
\tilde{V}  =\begin{pmatrix}
v_{0}\\
v_{1}\\
\vdots \\
v_{p-1}\\
v_p
\end{pmatrix}, \quad  \text{ } \quad
 U  =\begin{pmatrix}
a^{p}\\
a^{p-1}\\
\vdots \\
a\\
1
\end{pmatrix}  \quad  \text{ and } \quad
 \tilde{U}  =\begin{pmatrix}
a^{p}\\
a^{p-1}\\
\vdots \\
a\\
0
\end{pmatrix}
 $$
 on $\mathbb{C}^{p+1}$, here $ \otimes$ is the tensor product (or the Kronecker
 product). Equivalently
 
 $$ \left( I+\varepsilon ( K_{1}^{\rho} (\tilde{S})^{-1}U \otimes \tilde{U})\right)  \tilde{V}=0.$$
Since $\tilde{V}\neq 0$, this implies that $det \left( I+\varepsilon ( K_{1}^{\rho} (\tilde{S})^{-1} U \otimes \tilde{U})\right)=0$. But the spectrum of $\left( I+\varepsilon ( K_{1}^{\rho} (\tilde{S})^{-1}U \otimes \tilde{U})\right)$ is $\{1,  1+\varepsilon \langle  K_{1}^{\rho} (\tilde{S})^{-1}U \mid \tilde{U} \rangle  \}$, then we obviously have
$$ \langle  K_{1}^{\rho} (\tilde{S})^{-1}U \mid \tilde{U} \rangle=- \varepsilon \in\{-1,1\}.$$
On the other hand, $\langle  K_{1}^{\rho} (\tilde{S})^{-1}U \mid \tilde{U} \rangle$ viewed as a function of $\rho$ is continuous, so it must be constant. Let $r$ such that $0<r<\dfrac{1}{\rho}$, then $\Vert\dfrac{1}{r} S^{2p}_{2p+1}\Vert=\dfrac{1}{r}> \rho $, so $\dfrac{1}{r} S^{2p}_{2p+1} \notin C_{\rho}(\mathbb{C}^{2p+1})$. By definition of the $\rho$-numerical radius it implies that $w_\rho( S^{2p}_{2p+1})> r$. From Corollary 5.10 of \cite{CaSu} we infer that $w_\rho( S_{2p+1})^{2p} \geq w_\rho( S^{2p}_{2p+1}) > r$. By letting $r$ goes to $\dfrac{1}{\rho}$, we get
$$a(\rho)= \dfrac{1}{w_\rho( S_{2p+1})}\leq \rho^{\frac{1}{2p}}.$$
Thus, for all $k=1,\ldots, p$, we have,
$$0\leq \dfrac{a(\rho)^k}{\rho}\leq \dfrac{1 }{\rho^{1-\frac{k}{2p}}}\underset{ \rho \to +\infty}{\longrightarrow} 0.$$ 
Consequently,  
$$\dfrac{\rho}{a^{2p}}\langle  K_{1}^{\rho} (\tilde{S})^{-1}U \mid \tilde{U} \rangle \underset{ \rho \to +\infty}{\longrightarrow} 1.$$
 As $\rho K_{1}^{\rho} (\tilde{S})^{-1} \longrightarrow I$, $\dfrac{1}{a^{p}}U\longrightarrow e_0$ and $\dfrac{1}{a^{p}}  \tilde{U}  \longrightarrow e_0$, by letting again $\rho$ to $+ \infty$. By what we assert that $\langle  K_{1}^{\rho} (\tilde{S})^{-1}U \mid \tilde{U} \rangle$ is strictly positive and constant for $\rho$ sufficiently large, so equal $1$. We conclude that  $\varepsilon =-1$.
\end{proof}

A crucial point is now to study the nullity of the coefficients $v_k$. The following result shows that the situation differs according to the parity of the dimension.
\begin{theorem}\label{vk2}
	Let  $T$ is Harnack equivalent to $S$  in $C_{\rho}(\mathbb{C}^{n+1})$, $\rho>1$, we have
	\begin{enumerate}
		\item  If the dimension is an even number, i.e. $n+1=2p$, then $$\mathcal{N} (K_{z}^{\rho} (T))=\mathbb{C}(v_0,zv_1, \dots,z^{p-1}v_{p-1}, -z^{p}v_{p-1}, \dots,-z^{2p-1} v_0),\quad  \text{ for all }z\in \mathbb{T},$$
		with $v_k\neq 0$, for all $k=0,1,\ldots,p-1$.
		\item  If the dimension is an odd number, i.e. $n+1=2p+1$,   then $$\mathcal{N}  (K_{z}^{\rho} (T))=\mathbb{C}(v_0,zv_1, \dots,z^{p-1}v_{p-1},0, -z^{p}v_{p-1}, \dots,-z^{2p} v_0),\quad  \text{ for all }z\in \mathbb{T},$$
		with $v_k\neq 0$, for all $k=0,1,\ldots,p-1$.
	\end{enumerate}
\end{theorem}

 \begin{proof}By Theorem \ref{shift}, we have  $v_k=-v_{n-k}$, for all $k=0,1,\ldots,p-1$ and $v_0\neq 0$. Let us remark that if $n=2p$, then  $v_p=-v_{n-p}$, and in this case $v_p =0$. Now, to prove the theorem,  it suffices to show
that $v_k\neq 0$ for all $k=1,\ldots p-1$. Assume that there exists $m\in \{1,\ldots, p-1 \}$ such that $v_m =0$.

\textbf{Claim 1.} \textit{ We have} 
$$ det \left( K_{1}^{\rho} (\tilde{S})e_0, \ldots, K_{1}^{\rho} (\tilde{S})e_{m-1},U, K_{1}^{\rho} (\tilde{S})e_{m+1},\ldots, K_{1}^{\rho} (\tilde{S})e_{p-1} \right) =0,$$
\textit{where} $U=(a^{p-1},\ldots,a,1)^t$.

Suppose that $n=2p-1$, then we have
$$
0=\sum_{k=0}^{n}v_kK_{1}^{\rho} (S)e_k=\sum_{k=0}^{p-1}v_kK_{1}^{\rho} (S)e_k
-\sum_{k=p}^{2p-1}v_{2p-1-k}K_{1}^{\rho} (S)e_k. 
$$
Let $P$ be the orthogonal projection on the subspace spanned by $e_0, \cdots, e_{p-1}$.
Thus, we have
$$
0=\sum_{k=0}^{p-1}v_kPK_{1}^{\rho} (S)Pe_k
-\sum_{k=0}^{p-1}v_{k}PK_{1}^{\rho} (S)e_{2p-1-k}
=\sum_{k=0}^{p-1}v_kK_{1}^{\rho} (\tilde{S})e_k
-\sum_{k=0}^{p-1}v_{k}PK_{1}^{\rho} (S)We_{k}. 
$$
Taking into account that 
$$
\langle K_{1}^{\rho} (S)e_i\mid e_j \rangle
=\delta_{i,j}\rho+(1-\delta_{i,j})a^{\vert i-j\vert} \quad \text{ and } \quad WK_{1}^{\rho} (S)W=K_{1}^{\rho} (S),
$$
we obtain
\begin{equation}\label{dependent}
0= \sum_{k=0}^{p-1}v_kK_{1}^{\rho} (\tilde{S})e_k-a\langle\tilde{V} \mid U\rangle U.
\end{equation}
In a similar way, we can see that \eqref{dependent} remains valid when $n=2p$.
Then the family $\{K_{1}^{\rho} (\tilde{S})e_k,  k\in \{1,\ldots, p-1 \}\setminus \{m\}\} \cup \{U\}$ is linearly dependent, and hence Claim 1. follows.

\textbf{Claim 2.} \textit{For} $m \geq 2$\textit{, set}
$$ \tilde{D}_m(a)= \left|
\begin{array}{ccccc}
\rho & a & \cdots  & a^{m-1}& a^{m}\\
a & \rho & \ddots & \vdots & \vdots  \\
\vdots & \vdots & \ddots& a & a^{2}   \\
a^{m-1}& a^{m-2} & \cdots & \rho &a  \\
a^{m} &  a^{m-1} & \cdots & a & 1 
\end{array} \right|.
$$
\textit{Then we have} 
$$\tilde{D}_{m}(a)=\alpha \tilde{D}_{m-1}(a)-\beta \tilde{D}_{m-2}(a),$$ \textit{where} $\alpha= \rho+(\rho-2)a^2$, $\beta= a^2 (1-\rho)^2$, $\tilde{D}_{0}(a)=1$ and $\tilde{D}_{1}(a)=\rho-a^2$,
\textit{and the associated discriminant} $\Delta$ \textit{is given by} 
$$
\Delta=\left( a^2-1\right) ((a+1)\rho -2a )((a-1)\rho -2a ) .
$$

From Claim 1., we derive that
$$ D_p(a)=\left|
\begin{array}{cccccccc}
 \rho & a & \cdots  & a^{m-1}& a^{p-1}&  a^{m+1}&\cdots &  a^{p-1} \\
            a & \rho & \ddots & \vdots & \vdots & \vdots & \vdots &\vdots \\
            \vdots & \vdots & \ddots& a & a^{p-m+1} & a^3& \cdots & \vdots  \\
            a^{m-1}& a^{m-2} & \cdots & \rho &a^{p-m} &a^2 &\cdots & a^{p-l}  \\
              a^{m} &  a^{m-1} & \cdots & a & a^{p-m-1} & a & \ddots& \vdots \\
              a^{m+1} & a^{m} & \cdots & a^2& a^{p-m-2}&\rho & \ddots &\vdots \\
               \vdots &\vdots & \vdots & \vdots&\ddots & \vdots& \ddots & a \\
            a^{p-1} & a^{p-2} & \cdots & a^{p-m}& 1& a^{p-m-2}  &\cdots& \rho
\end{array} \right|=0.$$
 Denote by $C_k$ the $k$-th column, with $k\in \{m+1,\dots, p-1\}$. By replacing the  $k$-th column by $a^{p-1-k}C_k -C_m$ for $k=m+1,\dots,p-1$, we get
 $$D_p(a)=(\rho-1)^{p-1-m}a^{\frac{(p-2-m)(p-1-m)}{2}} a^{(m-p-1)}\tilde{D}_m(a)=(\rho-1)^{p-1-m}a^{\frac{(p-m)(p-m-1)}{2}} \tilde{D}_m(a).$$
Since, $\rho\neq 1$, then $D_p(a)=0$ if and only if $\tilde{D}_m(a)=0$. Now,
Multiplying the second column  of the  determinant $\tilde{D}_m(a)$ by $a$ and subtracting it from the first one gives
$$ \tilde{D}_m(a)= \left|
\begin{array}{ccccc}
\rho-a^2 & a & \cdots  & a^{m-1}& a^m\\
            a (1-\rho) & \rho & \ddots & \vdots & \vdots  \\
           \vdots  & \vdots & \ddots& a & a^2  \\
           0 & a^{m-2} & \cdots & \rho & a  \\
              0 &  a^{m-1} & \cdots & a & 1
\end{array} \right|.
 $$ By a similar operation with the rows, we find
$$ \tilde{D}_m(a)= \left|
\begin{array}{cccccc}
\rho-a^2- a^2(1-\rho)& a (1-\rho)& 0&\cdots  & 0& 0\\
a (1-\rho) & \rho & a &\cdots & a^{m-2} & a^{m-1}  \\
0  &  a & \rho  &\ddots& \vdots &  \vdots \\
\vdots  & \vdots &\ddots &\ddots& a & a^2  \\
0 & a^{m-2} & \cdots  &a & \rho & a  \\
0 &  a^{m-1} &\cdots  &a^2 & a & 1
\end{array} \right|.$$
Which implies that for $m\geq 2$, we have 
$$\tilde{D}_{m}(a)=\alpha \tilde{D}_{m-1}(a)-\beta \tilde{D}_{m-2}(a),$$ with $\alpha= \rho+(\rho-2)a^2$, $\beta= a^2 (1-\rho)^2$, $\tilde{D}_{0}(a)=1$ and $\tilde{D}_{1}(a)=\rho-a^2$. This recurrence relation form the equation 
\begin{equation}\label{equ:crac}
r^{2}-\alpha r+\beta=0,
\end{equation}
with discriminant
\begin{equation}\label{discriminant}
\Delta=\left( a^2-1\right) ((a+1)\rho -2a )((a-1)\rho -2a ) .
\end{equation} 
This ends proof of Claim 2..

Observe that $1<a=w_\rho(S_{n+1})^{-1}\leq \rho $ because the leading principal minor of order $2$ of the matrix associated with $K_{1}^{\rho}(S)$ is nonnegative. Thus we have $(a+1)\rho -2a >0$   and we can see that $\Delta$ is the same sign as $(a-1)\rho -2a$. Let  $\varphi(\rho)=\rho-2-2(a(\rho)-1)^{-1}$ where $a(\rho)= w_\rho(S_{n+1})^{-1}$. Clearly, $a(\rho)\longrightarrow 1^+$  if  $\rho \longrightarrow 1^+ $ and  $a(\rho)\longrightarrow +\infty$  if  $\rho \longrightarrow +\infty$. This implies that $\varphi(\rho)\longrightarrow -\infty$  if  $\rho \longrightarrow 1^+ $ and  $\varphi(\rho)\longrightarrow +\infty$  if  $\rho \longrightarrow +\infty$. Further, the function $\rho\longrightarrow a(\rho)$ is strictly decreasing on $\left[1 ,    +\infty\right)$. We conclude that  $\rho\longrightarrow \varphi(\rho)$ is strictly increasing on $\left[1 ,    +\infty\right)$. 
Consequently, there exists a unique $\rho_0=2+2(a(\rho_0)-1)^{-1}>2$ such that $\varphi(\rho_0)=0$. 
Therefore, there are three possibilities for the sign of $\Delta$. 

Now, using Claim 2., we distinguish three cases to end the proof of Theorem \ref{vk2}. 

\textbf{Case 1.} $\Delta=0$ if and only if  $ \rho =\dfrac{2a_0}{a_0-1}$. Equivalently,  $\rho_0=2+\dfrac{2}{a_0-1} $.

 For this value of $\rho$, the discriminant $\Delta$ admits a  root
$$ \lambda=\dfrac{\alpha}{2}=\dfrac{a_0(1+a_0)}{a_0-1}$$
with $a_0=a(\rho_0)=\dfrac{\rho_0}{\rho_0-2}$. So that
$$\tilde{D}_{m}(a)=(A+Bm) \lambda^{m},$$
where the constants $A$, $B$ can be determined from the ``initial conditions``
$$\tilde{D}_{0}(a)=1=A$$ and 
$$\tilde{D}_{1}(a)=\rho-a^2 =(1+B)\lambda.$$
This yields that
 $$A=1, \qquad  B = 1-a_0,$$
which gives,
$$ \tilde{D}_{m}(a_0)=(1+(1-a_0)m) \lambda^{m}.$$
On the other hand, for $\rho > 1$ and $n\geq 2$, set 
$$\mathbf{D}_n(a)= \left|
\begin{array}{ccccc}
\rho & a & \cdots  & a^{n-1}& a^{n}\\
a & \rho & \ddots & \vdots & \vdots  \\
\vdots & \vdots & \ddots& a & a^{2}   \\
a^{n-1}& a^{n-2} & \cdots & \rho &a  \\
a^{n} &  a^{n-1} & \cdots & a &  \rho 
\end{array} \right|.
$$
We now perform the same operation as before and get
$$\mathbf{D}_{n}(a)=\alpha \mathbf{D}_{n-1}(a)-\beta \mathbf{D}_{n-2}(a),$$ with  the ``initial conditions`` $\mathbf{D}_{0}(a)=\rho$ and $\mathbf{D}_{1}(a)=\rho^2-a^2$. So we can see that  this is the same recurrence relation for $\tilde{D}_m$. Hence for same values of $a_0, \rho_0$ and $\lambda$,  respectively, we have
$$\mathbf{D}_{n}(a)=(t+s n) \lambda^{n},$$
where the constants $t$, $s$ can be determined from the ``initial conditions``
$\mathbf{D}_{0}(a)=\rho_0=t$ and 
$\mathbf{D}_{1}(a_0)=\rho_0^2-a_0^2 =(t+s)\lambda$.
This yields that
$t=\rho_0$ and  $s= -a_0$,
which gives,
$$\mathbf{D}_{n}(a_0)=(\rho_0-a_0n) \lambda^{n}  .$$
If $n=2p-1$, we have $\mathbf{D}_{2p-1}(a_0)=0$, implies that $\rho_0 =2p+1$. But $\tilde{D}_{l}(a)=0$, for some $l\in \{1,\ldots, p-1 \}$,  is equivalent to
 $$l=\dfrac{1}{a_0-1}=\dfrac{\rho_0-2}{2}=p-\dfrac{1}{2} $$
 which is  a contradiction. Similarly, if $n=2p$, we have $\mathbf{D}_{2p}(a_0)=0$, implies that $\rho_0 =2pa_0$ and $a_0=1+\dfrac{1}{p}$. But in this case, $\tilde{D}_{l}(a_0)=0$, for some $l\in \{1,\ldots, p-1 \}$,  is equivalent to
 $$l=\dfrac{1}{a_0-1}=p, $$
 which is again a contradiction because $l\leq p-1$.
  We conclude that $v_k\neq 0$, for all $k=1,\ldots, p-1$. 
  
  \textbf{Case 2.} $\Delta>0$; in this case  $\rho>\rho_0>2$ ($\varphi$ is strictly increasing). So \eqref{equ:crac} admits two real roots $0<\lambda_1<\lambda_2$. 
  First, we claim that $0<\lambda_1 <1<\lambda_2$. Indeed, we have $\lambda_2>\dfrac{\rho}{2}>1$. If we put $\psi(r)=r^2-\alpha r+\beta$, we can see that $\psi(0)=\beta>0$ and $\psi(1)=-(1+a^2)<0$, so $0<\lambda_1<1$, since $\psi(\lambda)=0$.
Now, by  a similar operation on $\tilde{D}_{m}(a)$ as before, but this time with the last rows and the last line, we  find  
\begin{equation}\label{equ:key}
\tilde{D}_{m}(a)= [a^2(\rho -2)+1] \mathbf{D}_{m-1}(a)-a^2(\rho -1)^2\mathbf{D}_{m-2}(a).
\end{equation}
  Thus, if $\tilde{D}_{l}(a)=0$, for $l\in \{1,\ldots, p-1 \}$, then
  $$\mathbf{D}_{l}(a)=(\rho -1)\mathbf{D}_{l-1}(a).$$
  On the other hand,  we have
  $$\mathbf{D}_{m}(a)=t\lambda_1^{m}+s\lambda_2^{m},$$
  where the constants $t$, $s$  are determined by
  $$\mathbf{D}_{0}(a)=\rho=t+s \qquad \text{ and } \qquad \mathbf{D}_{n}(a)=t\lambda_1^{n}+s\lambda_2^{n}=0.$$
  This yields that
  $$t= \rho \dfrac{ \lambda_{2}^n }{\lambda_{2}^n-\lambda_{1}^n}, \qquad \text{ and } \qquad s= -\rho \dfrac{\lambda_{1}^n }{\lambda_{2}^n-\lambda_{1}^n},$$ 
  which gives,
  \begin{equation}\label{equ:key2}
  \mathbf{D}_{m}(a)=  \dfrac{ \rho }{\lambda_{2}^n-\lambda_{1}^n}(\lambda_{2}^n\lambda_1^{m}-\lambda_{1}^n\lambda_2^{m}).
  \end{equation}
  This implies that $$\mathbf{D}_{m+1}(a)-\mathbf{D}_{m}(a)=\dfrac{ \rho }{\lambda_{2}^n-\lambda_{1}^n}[\lambda_{2}^n\lambda_1^{m}(\lambda_1 -1)+\lambda_{1}^n\lambda_2^{m}(1-\lambda_2)].$$
  Since $\lambda_1 <1<\lambda_2$, then $ \mathbf{D}_{m+1}(a)<\mathbf{D}_{m}(a)$, for all $m=0,\ldots, p-1$. But we have
  $\mathbf{D}_{l}(a)=(\rho -1)\mathbf{D}_{l-1}(a)> \mathbf{D}_{l-1}(a)$ which leads to  a contradiction.  We also conclude that $v_k\neq 0$, for all $k=1,\ldots, p-1$.
  
  \textbf{Case 3.} $\Delta<0$; in this case $1<\rho <\rho_0$. So \eqref{equ:crac} admits two complex roots, $\lambda$ and its conjugate $\overline{\lambda}$. Let $\lambda =\lvert\lambda \rvert e^{i\omega}$ the root which lies in the upper half-plane ($0<\omega<\pi$). Clearly $\lvert\lambda \rvert=\beta=a(\rho-1)$, so $\lambda =a(\rho-1) e^{i\omega}$. Since
  $ \tilde{D}_{m}(a)=A \lambda^{m}+\overline{A} \,\, \overline{\lambda}^{m},$
  with 
  $1 =A+\overline{A}$  and $ \rho-a^2=A \lambda+\overline{A} \,\, \overline{\lambda}$.
 It follows that
   $$\tilde{D}_{m}(a)=\dfrac{  \overline{\lambda}-(\rho-a^2) }{\overline{\lambda} -\lambda }\lambda^{m}-\dfrac{  \lambda -(\rho-a^2)}{\overline{\lambda} -\lambda} \overline{\lambda}^{m}.$$
  Similarly; for $\mathbf{D}_{n}(a)$, we have
   $$ \mathbf{D}_{n}(a)=t \lambda^{n}+\overline{t}\, \, \overline{\lambda}^{n},$$
  such that 
  $$\mathbf{D}_{0}(a)=\rho =t+\overline{t}, \quad \mathbf{D}_{1}(a)=\rho^2-a^2=t \lambda_1+ \overline{t} \, \, \overline{\lambda}.$$
  which gives,
   $$\mathbf{D}_{n}(a)=\dfrac{ \rho \overline{\lambda}-(\rho^2-a^2) }{\overline{\lambda} -\lambda}\lambda^{n}-\dfrac{ \rho \lambda-(\rho^2-a^2)}{\overline{\lambda} -\lambda} \overline{\lambda}^{n}.$$
  Since $\mathbf{D}_{n}(a)=0$, we obtain that 
  $$ \left( \dfrac{\lambda }{\overline{\lambda}} \right)^{n}=  \dfrac{\rho \lambda-(\rho^2-a^2) }{\rho \overline{\lambda}-(\rho^2-a^2)}.$$
  Also
  $$ e^{i2n\omega}=  \dfrac{\rho \lambda-(\rho^2-a^2) }{\rho \overline{\lambda}-(\rho^2-a^2)}.$$
   It follows that
   $$ e^{i2n\omega}=  \dfrac{(\rho \lambda-(\rho^2-a^2))^2 }{\lvert\rho \overline{\lambda}-(\rho^2-a^2)\rvert^2}.$$
   
    Also, notice that \eqref{equ:key2} remains valid in this case, thus we have
    $$\mathbf{D}_{k}(a)=\rho a^k (\rho-1)^k\dfrac{\sin (n-k)\omega}{\sin n\omega}.$$
    Since $\omega \in ]0, \pi[$ and $\mathbf{D}_{k}(a)>0$ for 
    $k \in \left\lbrace 0,\cdots, n-1\right\rbrace $, we have necessarily 
    \begin{equation}\label{omega}
    \omega \in ]0, \frac{\pi}{n} [.
    \end{equation}
   Taking in account  that 
  \begin{equation}\label{cos_omega}
    \cos \omega=\dfrac{\alpha }{2\lvert\lambda \rvert}=\dfrac{\rho+(\rho-2)a^2  }{2a(\rho-1)} ,
    \end{equation} 
by a straightforward calculation, we obtain that $\lvert\rho \overline{\lambda}-(\rho^2-a^2)\rvert^2= a^2(\rho-1)(a^2-1)$. Then using \eqref{omega} we can deduce that
   \begin{equation}\label{equ:key_sin_1}
    \sin n\omega = \dfrac{\rho}{a}\sin \omega.
    \end{equation}
    
   Now,  if $\tilde{D}_{l}(a)=0$, for $l\in \{1,\ldots, p-1 \}$, $$ \left( \dfrac{\lambda }{\overline{\lambda}} \right)^{l}=  \dfrac{ \lambda-(\rho-a^2) }{ \overline{\lambda}-(\rho-a^2)}.$$
   Also
   $$ e^{i2l\omega}=  \dfrac{ \lambda-(\rho-a^2) }{ \overline{\lambda}-(\rho-a^2)}.$$
   It follows that
   $$ e^{i2l\omega}=  \dfrac{( \lambda-(\rho-a^2))^2 }{\lvert \overline{\lambda}-(\rho-a^2)\rvert^2}.$$
   Similarly as above, by a straightforward calculation and \eqref{omega}, we obtain that 
   $$ \sin l\omega =  \dfrac{\sqrt{\rho-1}}{\sqrt{a^2-1}}\sin \omega$$
   and
   $$ \cos l\omega =  \dfrac{\rho \sqrt{a^2-1}}{2a\sqrt{\rho-1}}.$$
   This two relations gives,
  \begin{equation}\label{equ:key_sin_2}
   \sin 2l\omega =  \dfrac{\rho}{a}\sin \omega.
  \end{equation}
    Let us remark that the relation \eqref{equ:key}, is also true in this case, by what it follows that
   $$\mathbf{D}_{l}(a)=(\rho -1)\mathbf{D}_{l-1}(a).$$
Therefore, we have
     \begin{equation}\label{equ:key_sin_3} a\sin (n-l)\omega =\sin (n-l+1)\omega
      \end{equation}
    Combines \eqref{equ:key_sin_1} and \eqref{equ:key_sin_2}, we derive that
    \begin{equation}\label{equ:key5}
    \sin n\omega = \sin (2l\omega),
    \end{equation}
    for some $l=1,\ldots, p-1$. Thus 
     $$\begin{cases} (n-2l)\omega= 2k\pi \text{ where } k \in \mathbb{Z}\cr
     \text{ or }\cr
    (n+2l)\omega= (2k'+1)\pi \text{ where } k'\in \mathbb{Z}.
    \end{cases} $$
    By \eqref{omega} we see that the first equation gives $2k\pi \in ]0, \pi[$ which is impossible. Again, by \eqref{omega} we observe that $ (2k'+1)\pi\in ]0, 2\pi[$, and obviously we have $k'=0$, then 
\begin{equation}\label{omega1}
    \omega = \frac{\pi}{n+2l}.
    \end{equation}
 
 For convenience, we set $q=n-l$ and using \eqref{cos_omega} and \eqref{equ:key_sin_3} we sucessively get
\begin{align*}
 &2 \frac{\sin(q+1)\omega}{\sin(q\omega)}(\rho-1)\cos\omega
 =\rho+(\rho-2)\frac{\sin^2 (q+1)\omega}{\sin^2 q\omega}\\
&\Leftrightarrow  2(\rho -1)\sin(q+1)\omega \left(\sin(q\omega)\cos \omega \right)
=\rho \sin^2 q\omega+(\rho-2)\sin^2 (q+1)\omega \\ 
&\Leftrightarrow (\rho-1)\sin^2 (q+1)\omega  + (\rho-1)\sin(q-1)\omega \sin(q+1)\omega\\  
&=\rho \sin^2 q\omega+(\rho-2)\sin^2 (q+1)\omega \\ 
&\Leftrightarrow \sin^2 (q+1)\omega+ (\rho-1) \sin(q+1)\omega \sin(q-1)\omega
=\rho \sin^2 q\omega \\
&\Leftrightarrow \cos (2q\omega)-\cos(2(q+1)\omega)=(\rho-1)\left[ 1-2\cos(2\omega)\right].
\end{align*}   
It gives
\begin{equation}\label{key_sin_4}
    \sin((2q+1)\omega)=(\rho-1)\sin \omega.
\end{equation}  
Now, combining \eqref{equ:key_sin_1}, \eqref{equ:key_sin_3} and \eqref{key_sin_4} we sucessively obtain   
\begin{align*}
 &\sin n\omega \sin(q+1)\omega=\sin q\omega \left[\sin \omega + \sin(2q+1)\omega \right]\\ 
&\Leftrightarrow \sin n\omega \sin(q+1)\omega=
\sin(q\omega)\left[2\sin(q+1)\omega \cos(q\omega) \right] \\
&\Leftrightarrow \sin n\omega \sin(q+1)\omega=\sin 2q\omega\sin(q+1)\omega.
\end{align*}      
Having in view \eqref{omega} we see that $\sin(q+1)\omega \neq 0$, and hence $\sin(n\omega)=\sin(2q\omega)$. Finally, we get  
$$\begin{cases} n\omega= 2q\omega +2k\pi \text{ where } k \in \mathbb{Z}\cr
     \text{ or }\cr
    n\omega= \pi -2q\omega +2k'\pi \text{ where } k'\in \mathbb{Z}.
    \end{cases} $$   
From the first equation and \eqref{omega}, we deduce that $-2k\pi=(n-2l)\omega \in ]0,\pi[$ which is impossible. According to the second equation and \eqref{omega} we have $(2k'+1)\pi=(3n-2l)\omega \in ]0,3\pi[$ which forces $k'$ to be $0$, thus we have
\begin{equation}\label{omega2}
\omega=\dfrac{\pi}{3n-2l}.
\end{equation}
The equalities \eqref{omega1} and \eqref{omega2} give $l=n/2$ which is absurd because 
$l \in \left\lbrace 1, \cdots , p-1 \right\rbrace$ and $n=2p-1$ or $n=2p$. This achieved the proof of
Theorem \ref{vk2}. 
\end{proof}
\section{Some consequences}\label{sec:3}
On the basis of the study made in section 2, we give some results which are of own interest. The first one is related to the critical value of $\rho_0>1$  which annulates the crucial discriminant given in \eqref{discriminant}.
\begin{proposition}\label{rho0}
Let $n \geq 1$. According to the notations used in the proof of the Theorem \ref{vk2}, we have $\rho_0=n+2$ and
$w_{n+2}(S_{n+1})=\frac{n}{n+2}$. In particular, we have
$$
w_{m+1}(T)\leq \frac{m-1}{m+1} \parallel T\parallel
$$
for any operator $T$ acting on a Hilbert space and such that $T^m=0$ ($m \geq 2$).
\end{proposition}
\begin{proof}
	Given the calculations made to determine $\mathbf{D}_{k}(a)$ in the first case ($\Delta=0$) of the proof of Theorem \ref*{vk2}, we have $\mathbf{D}_{k}(a_0)=(\rho_0-a_0k)\lambda^k$. Since 
	$\mathbf{D}_{n}(a_0)=0$, we get $n=\rho_0/a_0$ which leads to 
	$$
	\rho_0=n+2 \text{ and } a_0=\frac{1}{w_{\rho_0}(S_{n+1})}=\frac{n+2}{n},
	$$
	which proves the first statements. Then, the inequality given for nilpotent operators follows from Theorem 3.1 (or Corollary 4.1) in \cite{BaCa}.
\end{proof}
 It is not quite so easy to calculate or estimate $w_\rho(S_{n+1})$ (see for instance \cite{Carrot}). 
The next result shows that it can be done in a simpler way if $1<\rho<n+2$.
\begin{proposition}\label{wrho}
	Let $\rho \in ]1,n+2[$, then $x:=w_\rho(S_{n+1})=1/a$ is the unique solution in $]1/\rho, 1[$ of the following system of two equations
	\begin{equation}\label{omega_wrho}
	\begin{cases}
\dfrac{\sin(n\omega)}{\sin(\omega)} &=\rho x \\
	\cos (\omega) &=\dfrac{\rho x^2+(\rho-2)}{2x(\rho-1)}
	\end{cases}
	\end{equation}
	where $\omega \in ]0,\dfrac{\pi}{n}[$ is an intermediate variable uniquely determined. Moreover, the function
	$\rho\longmapsto \omega(\rho)$ is continuous and strictly decreasing on $\left] 1, n+2\right[ $ and we have $\omega(\left] 1, n+2\right[)=\left] 0, \frac{\pi}{n+1}\right[$.
\end{proposition}
\begin{proof}
The fact that $w_\rho(S_{n+1})$ is a solution of the system \eqref{omega_wrho} follows directly from the study made in the third case ($\Delta<0$) of the proof of Theorem \ref*{vk2}. Conversely, suppose that 
$x \in ]1/\rho, 1[$ is a solution, we easily see that we necessarily have 
$(\rho x^2+(\rho-2))/(2x(\rho-1))\in ]0,1[$, and then $\omega \in ]0,\pi[$ is uniquely determined by the second equation in \eqref{omega_wrho}, and in turn $x=w_\rho(S_{n+1})$ by the first equation
and \eqref{equ:key_sin_1}. Set $f(t)=\sin(nt)/\sin(t)$ for $t \in ]0,\pi/n[$, then we have $f'(t)=u(t)/(\sin^{2}(t))$ where
$u(t)=n\cos(nt)\sin(t)-\sin(nt)\cos(t)$. For any $t\in ]0,\pi/n[$, we see that $u'(t)=-(n^2-1)\sin(nt)\sin(t)<0$, thus the function $u$ is decreasing on $]0,\pi/n[$ and hence negative. Finally, we see that the function $f$ is strictly decreasing and continuous on $]0,\pi/n[$ and we have $f(\left] 0,\pi/n \right[ )=\left] 0,n \right[  $.
Now, we define the function $g$ on $\left[  1, +\infty \right[$ by setting $g(\rho)=\rho w_\rho(S_{n+1})$.
Suppose that there exist $\rho_1 $, $\rho_2$ in $\left]1,+\infty \right[ $ with $\rho_1<\rho_2$ and such that $g(\rho_1)=g(\rho_2)$, then using the facts that $s\mapsto h(s)=g(1+e^{s})$ is a convex function on the real line (see Theorem 4 of \cite{AnN}), and that $\lim_{s\rightarrow -\infty}h(s)=\lVert S_{n+1}\rVert=1$, we derive that $h$ is increasing on $\mathbb{R}$. Taking into account that $g(\rho_1)=g(\rho_2)$, we see that $h$  must be constant on $\left]-\infty, \ln(\rho_2-1) \right] $. Thus we should have $1=g(\rho_1)=g(\rho_2)$ which is impossible since $1<\rho w_\rho(S_{n+1})$ for any $\rho>1$ ($\rho^2-a^2=\mathbf{D}_{1}(a)>0$). Consequently, $g$ is a strictly increasing continuous function on $\left[  1, +\infty \right[$ with
$g(\left]1,n+2 \right[ )=\left]1,n \right[ $ ($ g(n+2)=n $ by Proposition \ref*{wrho} ).
In conclusion, $\omega=f^{(-1)}\circ g$ as a function of $\rho$ is strictly decreasing and continuous with $\omega(\left] 1, n+2\right[)=\left] 0, \frac{\pi}{n+1}\right[$ ($f^{(-1)}(1)=\pi/(n+1)$). This completes the proof.
\end{proof}
\begin{remark}
When $\rho=2$, Proposition \ref*{wrho} allows us to retrieve quickly the well known value of
$w_{2}(S_m)$ ($m \geq 2$). Set $n=m-1$, the second equation of \eqref{omega_wrho} gives
$x=\cos(\omega)$ and the first equation leads to $\sin(n\omega)=\sin(2\omega)$. Since $\omega \in ]0,\pi/n[$ the only possible answer is $\omega=\pi/(n+2)$ and hence
$$
w_{2}(S_m)=\frac{\pi}{m+1}.
$$
\end{remark}

\begin{proposition}\label{orbit} Let $\rho>0$  $(\rho\neq 1)$ and $U$ is a unitary matrix. Then, 
	$U^*SU$ is Harnack equivalent to $S$ in $ C_{\rho}(\mathbb{C}^{n+1})$ if and only if $Ue_k=\alpha e_k$, $\alpha \in \mathbb{T}$, for all $k$ such that $v_k\neq 0$. 
\end{proposition}
\begin{proof} By \cite[Corollary 2.24]{CaBeBel2018} we  know that, $U^*SU$ is Harnack equivalent to $S$ in $ C_{\rho}(\mathbb{C}^{n+1})$, $\rho>0$  $(\rho\neq 1)$, if and only if   $U(\mathcal{N} (K_{z}^{\rho} (S)))\subseteq \mathcal{N} (K_{z}^{\rho} (S))$, for all $z\in \mathbb{T}$, and the fact that
	$$\mathcal{N} (K_{z}^{\rho} (S))=\mathbb{C}V(z)=\mathbb{C}(v_0,zv_1, \dots,z^n v_n),\quad  \text{ for all }z\in \mathbb{T},$$ thus is equivalent to 
	\begin{equation}\label{VpU}  UV(z)=\alpha(z)V(z) ,\quad  \text{ for all }z\in \mathbb{T}.
	\end{equation}
	On the other hand,
	$$ \Vert V(1) \Vert =\Vert V(z) \Vert =\Vert UV(z) \Vert=\vert \alpha(z) \vert \Vert V(z) \Vert = \vert \alpha(z) \vert \Vert V(z) \Vert.$$ 
	So
	$$ \vert \alpha(z) \vert=1 .$$
	Using \eqref{VpU}, we get 
	$$ \alpha(z)\Vert V(1) \Vert=\langle UV(z), V(z)\rangle, ,\quad  \text{ for all }z\in \mathbb{T},$$
	thus the function $\alpha(.)$ is continuous from  $\mathbb{T}$ to the spectrum of $U$, so $\alpha(.)$ is a constant and  becomes
	\begin{equation}\label{alphaconstant}  UV(z)=\alpha V(z) , \quad  \text{ with }  \vert \alpha \vert=1     \text{ and  } z\in \mathbb{T}.
	\end{equation}
	We deduce that the space spanned by the vector $V(z)$ , $z\in \mathbb{T}$, is a subspace of $Ker(U-\alpha I)$, but this space is equal those spanned by $e_k$,  such that $v_k\neq 0$, since $v_k e_k=\int_{0}^{2\pi}e^{-ik\theta}V(e^{i\theta})dm$. 
\end{proof}
\begin{proposition}Let $\rho>1$. If $T \in C_{\rho}(\mathbb{C}^{n+1})$ is Harnack equivalent to $S$, then $Te_0=0$, $T^{*}e_n=0$ and $\langle  Te_n \mid e_0 \rangle=0$.
\end{proposition}
\begin{proof} By Theorem \ref{shift} and the inequality \eqref{Inequ:Harnack}, we obtain
	\begin{equation*}  K_{z}^{\rho} (T)V(z)=0, \quad \text{ for all } z\in \mathbb{T},
	\end{equation*}
	with  $V(z)=(v_0,zv_1, \dots,z^n v_n)$. Hence,
	\begin{equation}\label{equ:Harnack} 
	v_0 K_{e^{i\theta}}^{\rho} (T)e_0+ e^{i\theta} v_1 K_{e^{i\theta}}^{\rho} (T)e_1+ \dots + e^{n i\theta} v_n K_{e^{i\theta}}^{\rho} (T)e_n=0,
	\end{equation}
	for all $\theta \in \mathbb{R}$. Multiplying \eqref{equ:Harnack} by $1$ and $e^{i\theta}$, and integrating with respect the Haar measure $m$ on the torus, we get
	$$
	\rho v_0 e_0+  v_1 Te_1+ \dots + v_n  T^{n}e_n=0,
	$$
	and
	$$ 
	v_0 Te_0+  v_1 T^{2}e_1+ \dots + v_n  T^{n+1}e_n=0,
	$$
	By this two last equalities, we obtain
	$$
	(\rho -1)v_0 Te_0=0.
	$$
	Since $\rho\neq 1$ and $v_0 \neq 0$, we have necessarily $Te_0=0$.
	
	Now if we proceed as above we can prove that $T^{*}e_n=0$.
	On the other hand, if we multiply equation \eqref{equ:Harnack} by $e^{-i(n-1)\theta}$ and we integrate with respect to $m$ on the torus we obtain
	$$ 
	v_0 T^{\ast (n-1)}e_0+\ddots + v_{n-2} T^{\ast}e_{n-2}+ \rho v_{n-1} e_{n-1}+ v_n  Te_n=0.
	$$
	Then, taking the scalar product with the vector  $e_0$ and using the fact that $Te_0=0$, it comes $\langle  Te_n \mid e_0 \rangle =0$.
\end{proof}
By   Proposition \ref{orbit} and Theorem \ref{vk2}, we have
\begin{corollary} Let  $S\in C_{\rho}(\mathbb{C}^{n+1})$, $\rho>1$,  with $n+1$ is an even number. The unitary orbit of $S$ in the Harnack part of $S$ is trivial.
\end{corollary}
Moreover, any element in the Harnack part of $S$ is irreducible, when $n+1$ is an even number. More precisely, we have 
\begin{proposition} Let  $S\in C_{\rho}(\mathbb{C}^{n+1})$, $\rho>1$, where $n+1$ is an even number, i.e $n=2p-1$, $p\geq 1$. If $T \in C_{\rho}(\mathbb{C}^{n+1})$ is Harnack equivalent to $S$, then $T$ is irreducible.	
\end{proposition}
\begin{proof}Assume that $T$ is not irreducible. Then there exists an non trivial invariant subspace $E$, ($E\neq \emptyset $ and $E\neq \mathbb{C}^{n+1} $) of $T$ such that $ T= T_1 \oplus T_2$ according to the decomposition $\mathbb{C}^{n+1}=E\oplus E^{\bot}$. Also, $ K_{z}^{\rho} (T)= K_{z}^{\rho} (T_1)\oplus  K_{z}^{\rho} (T_2).$
Let $V(z)$ such that 	$V(z)=V_1(z)\oplus V_2(z)$. Put $\Omega_1=\{z\in \mathbb{T} : V_1(z)\neq 0\}$ and $\Omega_2=\{z\in \mathbb{T} : V_2(z)\neq 0\}$. Since $z\longmapsto V_1(z)$ and $z\longmapsto V_1(z)$ are continuous, we deduce that   $\Omega_1$ and $\Omega_2$ are open sets. Let $z_0\in \Omega_1 \cap \Omega_2$, we have $F=\mathbb{C}V_1(z_0)+\mathbb{C}V_2(z_0)\subset Ker(K_{z}^{\rho} (T))$. But $\dim(F) =2$ and $\dim (Ker(K_{z}^{\rho} (T)))=1 $, which is a contradiction. Hence $\Omega_1 \cap \Omega_2=\emptyset$. Since $ V(z)\neq 0$ for all $z\in \mathbb{T}$, we obtain $\Omega_1 \cup \Omega_2=\mathbb{T}$.  Now, the fact that $\mathbb{T}$ is connected, one of the open sets  $\Omega_1$ and $\Omega_2$ must be empty. This is a contraction with $E$ is not trivial.
\end{proof}
In the following result we give a complete description of an element   of $C_2 (\mathbb{C}^{n+1})$ Harnack equivalent to the $w_{2}$-normalized truncated shift $S$ with norm equal to those of $S$. 
\begin{theorem}\label{Harnack_C_2}
	Let $T$ be in $C_2 (\mathbb{C}^{n+1})$ such that $\lVert T\rVert=\lVert S\rVert=a=(\cos(\dfrac{\pi}{n+2}))^{-1}$. Assume that $T$ is Harnack equivalent to the $w_{2}$-normalized truncated shift $S$ of size $n+1$. Then the form of the matrix of $T$ depends on the parity of the dimension, more precisely we have :
	\begin{enumerate}\label{Harnack-parts}
		\item  If the dimension is an even number, i.e. $n+1=2p$, then $T=S$.
		\item  If the dimension is an odd number, i.e. $n+1=2p+1$,   then 
		\[ T=
		\begin{blockarray}{cccccccccc}
		Te_0 & \cdots &\cdots & Te_{p-1}& Te_p         & Te_{p+1}      &Te_{p+2} & \cdots&Te_{2p}\\
		\begin{block}{(ccccccccc)c}
		0      & a    &   0   & \cdots  &   \cdots     & \cdots        & \cdots  & \cdots &  0  \\
		\vdots &\ddots&\ddots &  \ddots &              &               &         &        &\vdots\\
		\vdots &      &\ddots & a       &  \ddots      &               &         &        &\vdots\\
		\vdots &      &       & \ddots  & ae^{i\theta} &   0           &         &        &\vdots\\
		\vdots &      &       &         & 0            & ae^{-i\theta} & \ddots  &        &\vdots\\
		\vdots &      &       &         &              &  \ddots       & a       & \ddots &\vdots\\
	    \vdots &      &       &         &              &               & \ddots  & \ddots & 0    \\
		\vdots &	  &       &	        &              &               &         & \ddots & a \\
		0      &\cdots&\cdots &	\cdots  &  \cdots      &  \cdots       &  \cdots & \cdots & 0  \\
		\end{block}
		\end{blockarray}
		\]
		where $\theta$ is an arbitrary real number.
	\end{enumerate}
\end{theorem}
\begin{proof}
	Let $T$ be an operator satisfying the assumptions of Theorem \ref{Harnack-parts}. 
	Applying Corollary 2.16 of \cite{CaBeBel2018} we get that the numerical range of $T$ is the closed unit disc $\overline{\mathbb{D}}$. Then using Theorem 1 of \cite{Wu} or Theorem 5.9 of \cite{GauWu}, we see that there exists a unitary operator $U\in B(\mathbb{C}^{n+1})$ such that $T=U^{\ast}SU$. Now, thanks to
	Theorem \ref{vk2} and Proposition \ref{orbit} we derive the desired result.
\end{proof}

We end this paper by proposing the following open questions.

{\it{Question 1}} \hspace{0.3 cm}  Can we remove the assumption $\lVert T\rVert=\lVert S\rVert$ in Theorem \ref{Harnack_C_2}?

Notice that the answer is positive when the dimension is two or three (see \cite[Theorems 3.1 and 3.3]{CaBeBel2018}).

And more generally, one can ask

{\it{Question 2}} \hspace{0.3 cm} Let $T$ be Harnack equivalent in $C_\rho (\mathbb{C}^{n+1})$, $\rho>1$, to the $w_{\rho}$-normalized truncated shift $S$ of size $n+1$. Is $T$ of the form given in Theorem \ref{Harnack_C_2} with respect to the parity of the dimension?

The second author wishes to note that the original idea of this paper is due
to the first author.
\bibliographystyle{amsplain}

\end{document}